\pdfoutput=1
\documentclass{amsart}
\usepackage[longtable]{multirow}
\usepackage{graphicx}
\usepackage{longtable}
\usepackage{amsmath}
\usepackage{amsthm}
\usepackage{amsfonts}
\usepackage{mathtools}
\usepackage{fullpage}
\usepackage{enumitem}

\usepackage{amssymb}
\theoremstyle{plain}
\newtheorem{theorem}{Theorem}
\newtheorem{lemma}{Lemma}
\newtheorem{observation}{Observation}

\newtheorem*{theorem*}{Theorem}
\theoremstyle{definition}
\newtheorem{definition}{Definition}
\newtheorem{remark}{Remark}

\newcommand*{\Id}[0]{{\operatorname{Id}}}

\usepackage{rotating}
\usepackage{accents}
\title[Orphan Calabi-Yau operator with arithmetic monodromy group]
{Orphan Calabi-Yau operator with arithmetic monodromy group}
\author{Tymoteusz Chmiel}
\thanks{The first author was supported by the National Science Center grant no. 2023/49/N/ST1/04089. Data sharing not applicable to this article as no datasets were generated or analysed during the current study.}
\begin{document}
\maketitle
\vspace{-8mm}

\begin{abstract}
We present an example of a Picard-Fuchs operator of a one-parameter family of Calabi-Yau threefolds which does not have a point of maximal unipotent monodromy, yet its monodromy group is of finite index in $\mathrm{Sp}(4,\mathbb{Z})$. In particular, it contains infinitely many maximally unipotent elements. We also state some related results for the remaining 17 double octic orphan operators.
\end{abstract}

\section*{Introduction}

The monodromy group is one of the most important invariants of a differential equation of Fuchsian type. If $\mathcal{P}$ is a Picard-Fuchs operator of a one-parameter family of Calabi-Yau threefolds, the monodromy group $Mon(\mathcal{P})$ is naturally realized as a subgroup of $\mathrm{Sp}(4,\mathbb{Z})$. This poses a question whether $Mon(\mathcal{P})$ is Zariski-dense in $\mathrm{Sp}(4,\mathbb{Z})$. If it is, one can further ask whether it is \emph{arithmetic}, i.e. of finite index, or \emph{thin}.

Indexes of several monodromy groups have been computed for operators with a singular point whose local monodromy has maximal unipotency index (see \cite{Brav-Thomas, Singh, Singh-Venkataramana, Hofmann-van Straten}). Existence of a point of maximal unipotent monodromy, or a \emph{MUM point}, is also crucial in the context of mirror symmetry.

In this paper we focus on operators which do not have a MUM singularity. They are called \emph{orphan operators}. Since every finite index subgroup of $\mathrm{Sp}(4,\mathbb{Z})$ contains maximally unipotent matrices, one could expect that monodromy groups of orphan operators are small in this sense.

We study orphan Picard-Fuchs operators of one-parameter families of double octic Calabi-Yau threefolds defined over $\mathbb{Q}$ (see \cite{Cynk-van Straten}). We construct rational bases using an approach which is easily applicable in general. Using rational realizations we study properties of the corresponding monodromy groups as linear groups generated by local monodromy operators. In particular, we obtain symplectic realizations for (subgroups of) double octic orphan monodromy groups.

The main result of this paper is the following theorem (Thm. \ref{th:MUM}, Thm. \ref{th:dense} and Thm. \ref{th:arithmetic}):

\begin{theorem*}\ 

\begin{enumerate}
	\item Monodromy groups of all double octic orphan operators are Zariski-dense in $\mathrm{Sp}(4,\mathbb{Z})$.
	\item Monodromy groups of double octic orphan operators contain maximally unipotent elements, with the~possible exception of operator \textbf{35}.
	\item Monodromy group of the double octic orphan operator \textbf{250} is of finite index in $\mathrm{Sp}(4,\mathbb{Z})$.
\end{enumerate}
\end{theorem*}
\noindent These surprising results show that the na\"ive expectation that the monodromy group of a Picard-Fuchs operator without a MUM point are small is false.

In section \ref{sec:general} we present basic facts concerning Picard-Fuchs operators. Section \ref{sec:ratbas} explains our method of finding rational bases for the monodromy action. Section \ref{sec:mongr} contains results obtained using this method when applied to orphan families of double octic Calabi-Yau threefolds. In section \ref{sec:symplectic} we construct symplectic bases for double octic orphan monodromy groups and show that they are Zariski-dense in $\mathrm{Sp}(4,\mathbb{Z})$. We also prove that the monodromy group of the operator \textbf{250} is arithmetic.

\section{Monodromy group of a Picard-Fuchs operator}\label{sec:general}

A \emph{Calabi-Yau threefold} is a smooth complex projective variety $X$ of dimension $3$ such that $\omega_X\simeq\mathcal{O}_X$ and $H^1(X,\mathcal{O}_X)=0$. By the Bogomolov-Tian-Todorov theorem, the Hodge number $h^{2,1}(X)$ is the dimension of the smooth deformation space of $X$. When $h^{2,1}(X)=1$, the universal deformation space is one-dimensional, i.e. there exists a (germ of a) smooth curve $\mathcal{S}$, a distinguished point $t_0\in\mathcal{S}$ and a family of smooth Calabi-Yau threefolds $\left(X_t\right)_{t\in\mathcal{S}}=\mathcal{X}\rightarrow\mathcal{S}$ such that $X\simeq X_{t_0}$. We say that $X$ \emph{deforms in a one-parameter family}.

Fix a family of complex volume elements $\omega_t\in H^{3,0}(X_t)$ depending holomorphically on $t$ and a locally constant family of $3$-cycles $\Delta_t\in H_3(X_t,\mathbb{Z})$. \emph{Period function} of the family $\mathcal{X}$ is
$y(t):=\int_{\Delta_t}\omega_t.$
It is defined in some neighbourhood $U$ of $t_0$. By the Hodge decomposition, $\dim H^3(X_s,\mathbb{C})=b_3(X_s)=2(h^{3,0}(X_s)+h^{2,1}(X_s))=4$ for all $s\in U$. Thus the elements
$$\left(\nabla_{\frac{\partial}{\partial t}}^4\omega_{t}\right)\Big\rvert_{t=s},\ 
\left(\nabla_{\frac{\partial}{\partial t}}^3\omega_{t}\right)\Big\rvert_{t=s},\ 
\left(\nabla_{\frac{\partial}{\partial t}}^2\omega_{t}\right)\Big\rvert_{t=s},\ 
\left(\nabla_{\frac{\partial}{\partial t}}\omega_t,\right)\Big\rvert_{t=s},\ \omega_s\in H^3(X_s,\mathbb{C})$$
are linearly dependent over $\mathbb{C}$.

It follows that the period function $y(t)$ satisfies a differential equation $\mathcal{P}y=0$ for some order four differential operator $\mathcal{P}$ with coefficients in $\mathcal{O}(U)$. This is the \textit{Picard-Fuchs operator} of the one-parameter family $\mathcal{X}\rightarrow\mathcal{S}$. By $\mathcal{S}ol(\mathcal{P},t_0)$ we denote the space of solutions of $\mathcal{P}=0$ in a neighbourhood of $t_0\in\mathcal{S}$. There is an isomorphism $\mathcal{S}ol(\mathcal{P},t_0)\simeq H^3(X_{t_0},\mathbb{C})$.

For $[\gamma]\in\pi_1(\mathcal{S},t_0)$ a solution $f\in\mathcal{S}ol(\mathcal{P},t_0)$ can be continued analytically along $\gamma$. This defines the \emph{monodromy representation} $Mon:\pi_1(\mathcal{S},t_0)\rightarrow\mathrm{Aut}\left(\mathcal{S}ol(\mathcal{P},t_0)\right)$. For a fixed basis $\mathcal{B}$ of $\mathcal{S}ol(\mathcal{P},s_0)$ we have the corresponding matrix representation $Mon^\mathcal{B}: \pi_1(\mathcal{S},t_0)\rightarrow\mathrm{GL}(4,\mathbb{C})$. We define $Mon(\mathcal{P}):=\operatorname{im}Mon$ and $Mon^\mathcal{B}(\mathcal{P}):=\operatorname{im}Mon^\mathcal{B}$.

Let $\mathcal{S}=\mathbb{P}^1\setminus\Sigma$ for some finite set $\Sigma$. The fundamental group $\pi_1(\mathbb{P}^1\setminus\Sigma)$ is generated by loops $\gamma_\sigma$ encircling singular points $\sigma\in\Sigma$. Consequently, the \emph{local monodromy operators} $M_\sigma:=Mon(\gamma_\sigma)$, $\sigma\in\Sigma$, generate the monodromy group. Operators $M_\sigma$ are quasi-unipotent, i.e. $(M_\sigma^k-\operatorname{Id})^4=0$ for some $k\in\mathbb{N}_{\geq 1}$ (see \cite{Landman}).

Picard-Fuchs operators are \emph{Fuchsian}: they have only regular singular points. Type of a singularity is determined by the Jordan form of its local monodromy. We have the following possibilities:
\begin{center}
	\begin{tabular}{|c|c|}
		\hline
		Type of singularity&Jordan form of local monodromy\\
		\hline
		MUM&
		$\begin{pmatrix}
		1&1&0&0\\0&1&1&0\\0&0&1&1\\0&0&0&1\\
		\end{pmatrix}$\\
		\hline
		$\tfrac{1}{n}$K&
		$\begin{pmatrix}
		1&1&0&0\\0&1&0&0\\0&0&\zeta_n&1\\0&0&0&\zeta_n\\
		\end{pmatrix}$\\
		\hline
		$\tfrac{1}{n}$C&
		$\begin{pmatrix}
		1&0&0&0\\0&\zeta_n&1&0\\0&0&\zeta_n&0\\0&0&0&\zeta_n^2\\
		\end{pmatrix}$\\
		\hline
		$F$&
		$\begin{pmatrix}
		1&0&0&0\\0&\zeta_{n_1}&0&0\\0&0&\zeta_{n_2}&0\\0&0&0&\zeta_{n_1}\zeta_{n_2}\\
		\end{pmatrix}$\\
		\hline
	\end{tabular}
\end{center}
Here $\zeta_m$ denotes some primitive root of unity of order $m$.

Singular points of type MUM are called \emph{points of maximal unipotent monodromy}. Singular points of type $\tfrac{1}{n}K$ are named due to their connection with $K3$ surfaces. Singular points of type $\tfrac{1}{n}C$ are also called \emph{conifold points} and are connected to (singular models of) rigid Calabi-Yau threefolds. Singular points of type $F$ are called \emph{finite singularities} and have local monodromy of finite order. For each point $\sigma$ of type $\tfrac{1}{n}C$ there exists the solution $f_c$, unique up to scaling, such that $\operatorname{im}(M_\sigma^n-\operatorname{Id})=\mathbb{C}\cdot f_c$. It is called the \emph{conifold period}.

An operator without a MUM singularity is called an \emph{orphan operator}.

\section{Rational basis for the monodromy action}\label{sec:ratbas}

An important disclaimer for all of the results presented in this paper is the following: \begin{center}
	\emph{all considered monodromy groups are determined numerically}.
\end{center}
This approach is common when studying Picard-Fuchs operators beyond the hypergeometric case (see \cite{Doran-Morgan, Hofmann, Hofmann-van Straten, Chmiel1, Chmiel2}). The relative error of our numerical approximations is smaller than $10^{-100}$. Numerical identification of the coefficients of the monodromy matrices yields the best results when $Mon^\mathcal{B}(\mathcal{P})\subset\mathrm{GL}(4,\mathbb{Q})$.

Let $\mathcal{P}$ be a Calabi-Yau operator and put $\mathcal{S}_\mathbb{Q}:=H^3(X_{t_0},\mathbb{Q})\hookrightarrow H^3(X_{t_0},\mathbb{Q})\otimes\mathbb{C}\simeq H^3(X_{t_0},\mathbb{C})\simeq \mathcal{S}ol(\mathcal{P},t_0)$. It is a monodromy-invariant rational subspace and with respect to any basis of $\mathcal{S}_\mathbb{Q}$ the monodromy matrices are rational.

We start with two general lemmas.

\begin{lemma}\label{lemma:mon-basis}
Assume that the monodromy group $Mon(\mathcal{P})$ acts irreducibly on $\mathcal{S}ol(\mathcal{P},t_0)$. Then for any solution \mbox{$0\neq f\in\mathcal{S}_\mathbb{Q}$} there exist monodromy operators \mbox{$M_1,\ldots,M_4\in Mon(\mathcal{P})$} such that
$$\mathcal{B}:=\Big\{M_1(f),M_2(f),M_3(f),M_4(f)\Big\}$$
is a a basis of $\mathcal{S}_\mathbb{Q}$.
\end{lemma}

\begin{proof}
Since $\mathcal{S}_\mathbb{Q}$ is monodromy-invariant, $\mathcal{B}\subset\mathcal{S}_\mathbb{Q}$. Assume that for all $M_1,\ldots,M_4\in Mon(\mathcal{P})$ the elements of $\mathcal{B}$ span inside $\mathcal{S}ol(\mathcal{P},t_0)$ a subspace of dimension $\leq 3$. Let $S$ be the $\mathbb{C}$-linear span of $M(f)$, $M\in Mon(\mathcal{P})$. The subspace $S$ is monodromy-invariant and $\dim S\leq 3$. But then $S\neq\mathcal{S}ol(\mathcal{P},t_0)$, which contradicts the fact that the monodromy group $Mon(\mathcal{P})$ acts irreducibly.
\end{proof}

Fix a non-zero solution $f\in\mathcal{S}_\mathbb{Q}$ and let $\mathcal{B}=\mathcal{B}(f)$ be as in the lemma. For $\alpha\in\mathbb{C}^*$ we have $\mathcal{B}(\alpha f)=\alpha\cdot\mathcal{B}(f)$, $Mon^{\mathcal{B}(f)}=Mon^{\mathcal{B}(\alpha f)}$ and $Mon^{\mathcal{B}(f)}(\mathcal{P})\subset\mathrm{GL}(4,\mathbb{Q})$ if and only if $Mon^{\mathcal{B}(\alpha f)}(\mathcal{P})\subset\mathrm{GL}(4,\mathbb{Q})$. Thus we can use Lemma \ref{lemma:mon-basis} when $f\not\in \mathcal{S}_\mathbb{Q}$ but $[f]\in \mathbb{P}\left(\mathcal{S}_\mathbb{Q}\right)\subset \mathbb{P}\left(\mathcal{S}ol(\mathcal{P},t_0)\right)$.

\begin{lemma}\label{lemma:unique}
Let $A:=\sum_{i=1}^{n} q_iM_i$, $q_i\in\mathbb{Q},\ M_i\in Mon(\mathcal{P})$. Assume that $\operatorname{rank}(A)=1$, resp. $\operatorname{corank}(A)=1$, and take $f\in\mathcal{S}ol(\mathcal{P},t_0)$ such that $\operatorname{im}(A)=\mathbb{C}\cdot f$, resp. $\operatorname{ker}(A)=\mathbb{C}\cdot f$. Then $\alpha f\in\mathcal{S}_\mathbb{Q}$ for some $\alpha\in\mathbb{C}^*$.
\end{lemma}	

\begin{proof}
$A$ acts as an endomorphism of $\mathcal{S}_\mathbb{Q}$. Let $g\in\mathcal{S}_\mathbb{Q}$ be such that \mbox{$\operatorname{im}\left(A|_{\mathcal{S}_\mathbb{Q}}\right)=\mathbb{Q}\cdot g$}, resp. \mbox{$\operatorname{ker}\left(A|_{\mathcal{S}_\mathbb{Q}}\right)=\mathbb{Q}\cdot g$}. Then $g\in\mathbb{C}\cdot f$.
\end{proof}

If $M_\sigma$ is a local monodromy operator around $\sigma\in\Sigma$ and $A=\sum_{i=0}^n a_iM_\sigma^i$, $a_i\in\mathbb{Q}$, the rank and the co-rank of $A$ can be computed by assuming that $M_\sigma$ is given in its Jordan form. In the construction of the Doran-Morgan basis in \cite{Doran-Morgan} one takes $A:=\left(M_0-\Id\right)^3$, where $0$ is a MUM point. Then $\operatorname{rank}(A)=1$ and the generator of $\operatorname{im}(A)$ is the unique holomorphic solution in a neighbourhood of $0$. It is equivalent to taking $A:=M_0-\Id$, since then $\operatorname{corank}(A)=1$ and the kernel of $A$ is spanned by the same solution.

Orphan operators do not have a point of maximal unipotent monodromy. Instead we use a singular point $s\in\Sigma$ of type $\tfrac{1}{n}C$. Put $A:=M_s^n-\Id$. The Jordan form of $M_s^n$ is
$$
\begin{pmatrix}
1&0&0&0\\
0&1&1&0\\
0&0&1&0\\
0&0&0&1\\
\end{pmatrix}
$$
and $\operatorname{rank}(A)=1$. The subspace $\operatorname{im}(A)$ is spanned by the conifold period $f_c$. By Lemma \ref{lemma:unique} we may assume $f_c\in\mathcal{S}_\mathbb{Q}$. Let $M_1,\cdots,M_4$ be as in the conclusion of Lemma \ref{lemma:mon-basis} applied to $f=f_c$ and put $f_i:=M_i(f)$. Fix a non-singular base point $t_0\in\Sigma$ and consider the matrix
$$
R:=\left(f_i^{(j-1)}(t_0)\right)_{i,j=1\ldots,4}
$$
The Frobenius method implies that the map
$$
\mathcal{S}ol(\mathcal{P},t_0)\ni g\mapsto \begin{pmatrix}g(t_0)\\g'(t_0)\\g''(t_0)\\g^{(3)}(t_0)\end{pmatrix}\in\mathbb{C}^4
$$
is an isomorphism. Thus $\det R\neq 0$ if and only if $f_1,\ldots,f_4$ are linearly independent. Condition $\det R\neq 0$ can be verified numerically, provided sufficient accuracy of the approximations used.

This way we obtain potential generating sets $\left\{M_1(f_c),\cdots,M_4(f_c)\right\}$ and an efficient test on whether they really form bases of the rational subspace $\mathcal{S}_\mathbb{Q}$. It is usually not hard to find appropriate $M_1,\cdots,M_4$ such that $\det R\neq 0$. In these cases the inequality can be proven (not only observed numerically) using precise bounds on the approximation error.

Results of this approach applied to \emph{double octic orphan operators} are described in the next section.

\section{Monodromy groups of orphan operators}\label{sec:mongr}

A \textit{double octic} is a smooth Calabi-Yau threefold $X$ obtained as a resolution of singularities of a double cover $\overline{X}\rightarrow\mathbb{P}^3$ branched along a union of eight planes. A \textit{double octic operator} is a Picard-Fuchs operator of a one-parameter family of double octics $X_t$ defined over $\mathbb{Q}$ and satisfying $h^{2,1}(X_t)=1$. Double octics $X$ defined over $\mathbb{Q}$ such that $h^{2,1}(X)\leq 1$ were classified in \cite{Cynk-Kocel} and the Picard-Fuchs operators for those with $h^{2,1}(X)=1$ are known (see \cite{Chmiel1, Cynk-van Straten}).

Monodromy groups of double octic operators with a MUM point were found in \cite{Chmiel1} using a generalization of the Doran-Morgan method. The focus of the present paper is on double octic operators which do not have a point of maximal unipotent monodromy. There are 25 double octic orphan operators. A complete list can be found in \cite{Cynk-van Straten}.

Seven of them are of order 2 and the monodromy action is not irreducible. These examples are less interesting since the associated local system decomposes into a tensor product of a constant system and a well-understood system of rank $2$ (see \cite{GvG, Rohde}).

The remaining 18 double octic orphan operators are of order 4. Corresponding arrangements of eight planes are listed in \cite{Cynk-van Straten}; we follow the numbering from this paper. The list of orphan arrangements, with arrangements yielding equivalent Picard-Fuchs operators listed together, is as follows:
$$
\begin{tabular}{l}
\textbf{33},\ 70 \\
\textbf{35},\ 71,\ 247,\ 252 \\
\textbf{97},\ 98 \\
\textbf{152},\ 198 \\
\textbf{153},\ 197 \\
\textbf{243} \\
\textbf{248} \\
\textbf{250},\ 258 \\
\textbf{266},\ 273 \\
\end{tabular}
$$

Now we present rational realizations of monodromy groups of double octic orphan operators. Let $\mathcal{P}$ be such an operator. The following list contains local monodromy operators $M_\sigma$ for singular points $\sigma\in\Sigma$ written in a rational basis $\mathcal{B}$. These bases were found using the procedure from the previous section. For all operators except \textbf{266} the basis $\mathcal{B}$ is of the form
$$\big\{f_c,N_{\sigma_1}(f_c),N_{\sigma_2}(N_{\sigma_1}(f_c)),N_{\sigma_3}(N_{\sigma_2}(N_{\sigma_1}(f_c)))\big\},$$
where $\sigma_i\in\Sigma$, $N_{\sigma_i}:=M_{\sigma_i}-\operatorname{Id}$ and $f_c$ is the conifold period associated with a singularity $s\in\Sigma$ of type $\tfrac{1}{n}$C. The basis $\mathcal{B}$ is then given in the format $[s;M_{\sigma_1},M_{\sigma_2},M_{\sigma_3}]$.

\medskip

\textbf{Operator 33}

\medskip

$\mathcal{B}=[1;M_0,M_2,M_0]$

\medskip

$
M_{1}=
\begin{pmatrix}
-1&-4&-24&16\\
0&2&2&4\\
0&-1&-1&-4\\
0&-\tfrac{1}{4}&-\tfrac{1}{2}&0\\
\end{pmatrix}
$
$
M_{0}=
\begin{pmatrix}
1&0&0&0\\
1&1&0&0\\
0&0&1&0\\
0&0&1&1\\
\end{pmatrix}
$
$
M_{2}=
\begin{pmatrix}
1&0&0&0\\
0&1&0&0\\
-\tfrac{1}{2}&1&1&12\\
0&0&0&1\\
\end{pmatrix}
$

\pagebreak

\textbf{Operator 35}

\medskip

$\mathcal{B}=[0;M_1,M_{-1},M_1]$

\medskip

$
M_{0}=
\begin{pmatrix}
-1&-24&-32&-512\\
0&3&-8&32\\
0&\tfrac{1}{2}&-1&8\\
0&-\tfrac{1}{8}&\tfrac{1}{2}&-1\\
\end{pmatrix}
$
$
M_{1}=
\begin{pmatrix}
1&0&0&0\\
1&1&0&0\\
0&0&1&0\\
0&0&1&1\\
\end{pmatrix}
$
$
M_{-1}=
\begin{pmatrix}
1&0&0&0\\
0&1&0&0\\
\tfrac{1}{4}&1&1&32\\
0&0&0&1\\
\end{pmatrix}
$

\medskip

\textbf{Operator 97}

\medskip

$\mathcal{B}=[-2;M_{-1},M_{0},M_{-1}]$

\medskip

$
M_{0}=
\begin{pmatrix}
-3&0&0&-32\\
2&1&0&16\\
-\tfrac{1}{2}&1&-1&-8\\
\tfrac{1}{4}&0&0&3\\
\end{pmatrix}
$
$
M_{-2}=
\begin{pmatrix}
1&4&-16&0\\
0&1&0&0\\
0&0&1&0\\
0&0&0&1\\
\end{pmatrix}
$
$
M_{-1}=
\begin{pmatrix}
1&0&0&0\\
1&1&0&0\\
0&0&1&0\\
0&0&1&1\\
\end{pmatrix}
$

\medskip

\textbf{Operator 152}

\medskip

$\mathcal{B}=[0;M_{-1},M_{1},M_{-1}]$

\medskip

$
M_{0}=
\begin{pmatrix}
-1&-4&32&-64\\
0&1&0&0\\
0&\tfrac{1}{2}&-1&0\\
0&\tfrac{1}{8}&-\tfrac{1}{2}&1\\
\end{pmatrix}
$
$
M_{1}=
\begin{pmatrix}
-3&0&0&-128\\
1&1&0&32\\
-\tfrac{1}{4}&1&-1&-16\\
\tfrac{1}{16}&0&0&3\\
\end{pmatrix}
$
$
M_{-1}=
\begin{pmatrix}
1&0&0&0\\
1&1&0&0\\
0&0&1&0\\
0&0&1&1\\
\end{pmatrix}
$

\medskip

\textbf{Operator 153}

\medskip

$\mathcal{B}=[-1;M_{0},M_{-2},M_{0}]$

\medskip

$
M_{-1}=
\begin{pmatrix}
1&0&0&0\\
1&1&0&32\\
0&0&1&0\\
0&-\tfrac{1}{8}&1&-3\\
\end{pmatrix}
$
$
M_{0}=
\begin{pmatrix}
-1&-68&448&-2368\\
0&-\tfrac{91}{3}&\tfrac{470}{3}&-\tfrac{3008}{3}\\
0&\tfrac{4}{3}&-\tfrac{17}{3}&\tfrac{128}{3}\\
0&\tfrac{9}{8}&-\tfrac{45}{3}&37\\
\end{pmatrix}
$
$
M_{-2}=
\begin{pmatrix}
-9&0&0&-512\\
-\tfrac{25}{6}&1&0&-\tfrac{640}{3}\\
\tfrac{1}{6}&1&-1&\tfrac{112}{3}\\
\tfrac{5}{32}&0&0&9\\
\end{pmatrix}
$

\medskip

\textbf{Operator 243}

\medskip

$\mathcal{B}=[1;M_{\infty}^2,M_{\frac{3}{2}},M_{\infty}^2]$

\medskip

$
M_{1}=
\begin{pmatrix}
-1&-204&360&-14880\\
0&\tfrac{43}{4}&-\tfrac{65}{2}&650\\
0&\tfrac{9}{10}&-2&60\\
0&-\tfrac{21}{160}&\tfrac{7}{16}&-\tfrac{31}{4}\\
\end{pmatrix}
$
$
M_{2}=
\begin{pmatrix}
-2&-12&-120&-1440\\
-\tfrac{9}{16}&-\tfrac{5}{4}&-\tfrac{45}{2}&-270\\
\tfrac{3}{49}&\tfrac{3}{10}&4&36\\
\tfrac{3}{640}&\tfrac{3}{160}&\tfrac{3}{16}&\tfrac{13}{4}\\
\end{pmatrix}
$
$
M_{\frac{3}{2}}=
\begin{pmatrix}
1&0&0&0\\
0&1&0&0\\
\tfrac{1}{20}&1&1&88\\
0&0&0&1\\
\end{pmatrix}
$
$
M_{\infty}=
\begin{pmatrix}
-1&0&0&0\\
-\tfrac{1}{2}&-1&0&0\\
0&0&-1&0\\
0&0&-\tfrac{1}{2}&-1\\
\end{pmatrix}
$

\pagebreak

\textbf{Operator 248}

\medskip

$\mathcal{B}=[-\frac{1}{2};M_{-1},M_{0},M_{-1}]$

\medskip

$
M_{0}=
\begin{pmatrix}
1&0&0&0\\
0&1&0&0\\
-\tfrac{1}{2}&1&1&16\\
0&0&0&1\\
\end{pmatrix}
$
$
M_{-\tfrac{1}{2}}=
\begin{pmatrix}
1&4&16&32\\
0&1&0&0\\
0&0&1&0\\
0&0&0&1\\
\end{pmatrix}
$
$
M_{-1}=
\begin{pmatrix}
1&0&0&0\\
1&1&0&0\\
0&0&1&0\\
0&0&1&1\\
\end{pmatrix}
$
$
M_{-\tfrac{3}{2}}=
\begin{pmatrix}
3&4&32&32\\
-1&-1&-16&-16\\
0&0&1&0\\
0&0&0&1\\
\end{pmatrix}
$
$
M_{-2}=
\begin{pmatrix}
17&16&128&0\\
-8&-7&-64&0\\
-1&-1&-7&0\\
\tfrac{1}{2}&\tfrac{1}{2}&4&1\\
\end{pmatrix}
$

\medskip

\textbf{Operator 250}

\medskip

$\mathcal{B}=[1;M_{\infty},M_{-2},M_{\infty}]$

\medskip

$
M_{1}=
\begin{pmatrix}
1&-8&64&-128\\
0&1&0&0\\
0&0&1&0\\
0&0&0&1\\
\end{pmatrix}
$
$
M_{\infty}=
\begin{pmatrix}
1&-4&0&0\\
1&-3&0&0\\
0&0&1&-4\\
0&0&1&-3\\
\end{pmatrix}
$
$
M_{-1}=
\begin{pmatrix}
-1&24&-80&224\\
0&7&-24&80\\
\tfrac{3}{8}&\tfrac{3}{4}&-5&28\\
\tfrac{1}{16}&-\tfrac{3}{8}&1&-1\\
\end{pmatrix}
$
$
M_{-2}=
\begin{pmatrix}
1&0&0&0\\
0&1&0&0\\
-\tfrac{1}{2}&1&1&-8\\
0&0&0&1\\
\end{pmatrix}
$

\medskip

\textbf{Operator 266}

\medskip

$\mathcal{B}=\left\{f_c,N_{-\frac{1}{2}}(f_c),N_{-\frac{1}{4}}(f_c),N_{0}(f_c)\right\}$; $f_c$ is the conifold period at $-1$

\medskip

$
M_{-\tfrac{1}{4}}=
\begin{pmatrix}
1&0&0&0\\
0&1&0&0\\
1&-6&1&-3\\
0&0&0&1\\
\end{pmatrix}
$
$
M_{-\tfrac{1}{2}}=
\begin{pmatrix}
1&-\tfrac{12}{5}&\tfrac{24}{5}&\tfrac{12}{5}\\
1&-\tfrac{11}{5}&-\tfrac{8}{5}&-\tfrac{24}{5}\\
0&\tfrac{3}{5}&\tfrac{1}{5}&-\tfrac{3}{5}\\
0&-\tfrac{2}{5}&\tfrac{4}{5}&\tfrac{7}{5}\\
\end{pmatrix}
$
$
M_{-1}=
\begin{pmatrix}
1&12&-16&0\\
0&1&0&0\\
0&0&1&0\\
0&0&0&1\\
\end{pmatrix}
$
$
M_{0}=
\begin{pmatrix}
1&\tfrac{96}{5}&-\tfrac{48}{5}&0\\
0&\tfrac{13}{5}&-\tfrac{4}{5}&0\\
0&\tfrac{36}{5}&-\tfrac{13}{5}&0\\
1&\tfrac{16}{5}&-\tfrac{28}{5}&-1\\
\end{pmatrix}
$
$
M_{\tfrac{1}{2}}=
\begin{pmatrix}
5&-12&-16&-12\\
0&1&0&0\\
1&-3&-3&-3\\
0&0&0&1\\
\end{pmatrix}
$

\section{Symplectic basis for the monodromy action}\label{sec:symplectic}

There exists a monodromy-invariant lattice $\mathcal{S}_\mathbb{Z}:=H^3(X_{t_0},\mathbb{Z})\subset H^3(X_{t_0},\mathbb{Q})=\mathcal{S}_\mathbb{Q}$. The cup product
$$
\cup:H^3(X_{t_0},\mathbb{Z})\times H^3(X_{t_0},\mathbb{Z})\rightarrow H^6(X_{t_0},\mathbb{Z})\simeq\mathbb{Z}
$$
defines a non-degenerate skew-symmetric bilinear form on $\mathcal{S}_\mathbb{Z}$. It follows that there exists a \textit{symplectic basis} of $\mathcal{S}ol(\mathcal{P},t_0)$, i.e. a basis $\mathcal{B}$ such that $Mon^\mathcal{B}(\mathcal{P})\subset \mathrm{Sp}(4,\mathbb{Z})$.

Monodromy group of a hypergeometric operator is generated by local monodromies around a MUM point and a conifold point. In the Doran-Morgan basis the monodromy matrices of these generators are
$$
\begin{pmatrix}
1&0&0&0\\
1&1&0&0\\
0&d&1&0\\
0&0&1&1\\
\end{pmatrix}\quad\textnormal{ and }\quad
\begin{pmatrix}
1&-k&-1&-1\\
0&1&0&0\\
0&0&1&0\\
0&0&0&1\\
\end{pmatrix}
$$
for some $d,k\in\mathbb{Z}$. They can be brought to the symplectic group $\mathrm{Sp}(4,\mathbb{Z})$ by conjugating with the matrix
$$
\begin{pmatrix}
0&-1&0&0\\
-1&0&0&0\\
0&-d&-1&0\\
0&k&1&1\\
\end{pmatrix}
$$
Corresponding generators of the monodromy group are
$$
\begin{pmatrix}
1&1&0&0\\
0&1&0&0\\
d&d&1&0\\
0&-k&-1&1\\
\end{pmatrix}\quad\textnormal{ and }\quad
\begin{pmatrix}
1&0&0&0\\
0&1&0&1\\
0&0&1&0\\
0&0&0&1\\
\end{pmatrix}
$$
The monodromy around a conifold singularity is a symplectic reflection.

Given a symplectic basis $\mathcal{B}$ for $Mon(\mathcal{P})$, interesting arithmetical questions about the monodromy group $Mon^\mathcal{B}(\mathcal{P})\subset\mathrm{Sp}(4,\mathbb{Z})$ arise. The first one is whether $Mon^\mathcal{B}(\mathcal{P})$ is arithmetic or thin.

\begin{definition}
Let $G$ be a finitely generated, Zariski-dense subgroup of $\mathrm{Sp}(4,\mathbb{Z})$. $G$ is called \emph{arithmetic} if the index $\left[\mathrm{Sp}(4,\mathbb{Z}):G\right]$ is finite. Otherwise, it is called \emph{thin}.
\end{definition}

\begin{remark}
If $G\subset\mathrm{Sp}(4,\mathbb{Z})$ is a subgroup such that $[\mathrm{Sp}(4,\mathbb{Z}):\overline G]<+\infty$, then $\overline{G}=\mathrm{Sp}(4,\mathbb{Z})$. Indeed, the closure $\overline{G}$ is a Zariski-clopen subgroup of $\mathrm{Sp}(4,\mathbb{Z})$. In particular, if $G$ is of finite index, it is arithmetic.
\end{remark}

Question of arithmeticity is difficult even for hypergeometric operators. There are $14$ hypergeometric operators and their symplectic monodromy group depends on two parameters $(d,k)\in\mathbb{N}^2$ as above. It is arithmetic for $(d,k)=$ (1,2), (1,3), (2,3), (3,4), (4,4), (6,5), (9,6) and thin for the remaining pairs \mbox{$(d,k)=$ (1,4), (2,4), (4,5), (5,5), (8,6), (12,7), (16,8)} (see \cite{Brav-Thomas,Singh,Singh-Venkataramana}).

A necessary condition for arithmecity is given by the following:

\begin{observation}\label{ob:index_MUM}
Let $G\subset \mathrm{Sp}(4,\mathbb{Z})$ be a subgroup which does not contain a maximally unipotent matrix. Then the index $[\mathrm{Sp}(4,\mathbb{Z}):G]$ is infinite.
\end{observation} 

\begin{proof}
Let
$$
M:=\begin{pmatrix}
1&1&0&0\\
0&1&1&0\\
0&0&1&1\\
0&0&0&1\\
\end{pmatrix}
$$
be a maximally unipotent matrix in its Jordan form. For all $k\in\mathbb{N}$ we have
$$
M^k=\begin{pmatrix}
1&k&\binom{k}{2}&\binom{k}{3}\\
0&1&k&\binom{k}{2}\\
0&0&1&k\\
0&0&0&1\\
\end{pmatrix}
$$
It follows that a power of a maximally unipotent matrix is again maximally unipotent. If $A$ is a maximally unipotent element of $\mathrm{Sp}(4,\mathbb{Z})$, we conclude that $A^k\not\in G$ for all $k\in\mathbb{N}_{>0}$ and the cosets $A^kG$ are different.
\end{proof}	

Local monodromies of orphan operators, generating the monodromy group, are not maximally unipotent. If the monodromy group does not contain \textit{any} maximally unipotent elements, its index in $\mathrm{Sp}(4,\mathbb{Z})$ is infinite. However, we have the following result:

\begin{theorem}\label{th:MUM}
Let $\mathcal{P}$ be a double octic orphan operator, other than operator \textbf{35}. The monodromy group $Mon(\mathcal{P})$ contains a maximally unipotent element.
\end{theorem} 

\begin{proof}
We have the following maximally unipotent elements in orphan monodromy groups:

\pagebreak

\textbf{Operator 33}

$\left(M_2\cdot M_0\cdot M_1^{-1}\right)^{-2}=
\begin{pmatrix}
81&-184&240&-2848\\
34&-77&92&-1128\\
-23&52&-63&768\\
-2&\tfrac{9}{2}&-5&63\\
\end{pmatrix}
$

\textbf{Operator 97}

$
\left(M_{-2}\cdot M_{0}^{-1}\cdot M_{-2}^{-1}\cdot M_{-1}^{-1}\cdot M_{-2}\cdot M_{-2}\cdot M_0\right)^{-1}=
\begin{pmatrix}
157&-772&2704&1920\\
-36&181&-624&-448\\
-10&52&-175&-128\\
-13&64&-224&-159\\
\end{pmatrix}
$

\textbf{Operator 152}

$
\left(M_{1}^{-1} \cdot M_{-1}^{1} \cdot M_{-1}^{-1} \cdot M_{1} \cdot M_{1} \cdot M_{0} \cdot M_{0} \cdot M_{-1}\right)^{-1}=
\begin{pmatrix}
-1031&-8&-992&-32768\\
1032&9&992&32768\\
-32&0&-31&-1024\\
33&0&32&1057\\
\end{pmatrix}
$

\textbf{Operator 153}

$
\left(M_{-1}^{-1} \cdot M_{0} \cdot M_{-2} \cdot M_{0}^{-1} \cdot M_{0}^{-1} \cdot M_{-1}^{-1} \cdot M_{0}^{-1} \cdot M_{-2}\right)^{-1}=
\begin{pmatrix}
2673&-14016&75072&-315392\\
\tfrac{3496}{3}&-5759&32576&-129024\\
-\tfrac{196}{3}&352&-1839&7936\\
-\tfrac{89}{2}&220&-1244&4929\\
\end{pmatrix}
$

\textbf{Operator 243}

$
M_{\infty} \cdot M_{2}^{-1} \cdot M_{1} \cdot M_{1} \cdot M_{\infty}^{-1} \cdot M_{1}^{-1} \cdot M_{1}^{-1}=
\begin{pmatrix}
-578&167628&-597720&11175840\\
-\tfrac{5937}{16}&\tfrac{430453}{4}&-\tfrac{767385}{2}&7174590\\
\tfrac{87}{8}&-\tfrac{31539}{10}&11248&-210276\\
\tfrac{783}{128}&-\tfrac{283851}{160}&\tfrac{101223}{16}&-\tfrac{473117}{4}\\
\end{pmatrix}
$

\textbf{Operator 248}

$
\left(M_{-\frac{3}{2}}^{-1} \cdot M_{0} \cdot M_{-\frac{1}{2}} \cdot M_{-\frac{3}{2}}^{-1}\right)^{-1}=
\begin{pmatrix}
-23&-28&-304&-96\\
-6&-7&-80&-64\\
\tfrac{5}{2}&3&33&16\\
0&0&0&1\\
\end{pmatrix}
$

\textbf{Operator 250}

$
M_{-1}^{-1} \cdot M_{-2}^{-1} \cdot M_{-2}^{-1} \cdot M_{-1}^{-1} \cdot M_{1}=
\begin{pmatrix}
-175&1592&-10560&16256\\
-28&257&-1696&2624\\
\tfrac{9}{2}&-40&269&-408\\
\tfrac{15}{4}&-34&226&-347\\
\end{pmatrix}
$

\textbf{Operator 266}

$
M_{-\frac{1}{4}} \cdot M_{0} \cdot M_{\frac{1}{2}} \cdot M_{-\frac{1}{4}} \cdot M_{-\frac{1}{2}}=
\begin{pmatrix}
-\tfrac{2927}{5}&624&\tfrac{19208}{5}&\tfrac{23544}{5}\\
-\tfrac{61}{5}&13&\tfrac{404}{5}&\tfrac{492}{5}\\
-\tfrac{197}{5}&42&\tfrac{1293}{5}&\tfrac{1584}{5}\\
-\tfrac{197}{5}&42&\tfrac{1288}{5}&\tfrac{1589}{5}\\
\end{pmatrix}
$
\end{proof}

\begin{remark}
The Doran-Morgan construction can be applied to operators with a point of maximal unipotent monodromy and a conifold point. Generalization of this method from \cite{Chmiel1} requires only \emph{some} maximally unipotent monodromy element in the monodromy group. Theorem \ref{th:MUM} shows that for double octic orphan operators, other than \textbf{35}, one can construct such generalized Doran-Morgan bases despite the lack of a MUM singularity.
\end{remark}

To find a symplectic basis, first we find the invariant skew-symmetric form. Let $G:=Mon^\mathcal{B}(\mathcal{P})\subset\mathrm{GL}(4,\mathbb{Q})$ be a rational realization of the monodromy group with generators $M_1,\ldots,M_n$. Let $S=\left(s_{i,j}\right)_{i,j=1,\ldots,4}$ be a generic antisymmetric non-singular matrix. Let $\mathrm{S}=(\mathrm{s}_{i,j})_{i,j=1,\ldots,4}$ be a solution of the following system of $16n$ linear equations:
\begin{equation}\label{eq:system}
\big(M_k\cdot S\cdot M_k^T-S\big)_{i,j}=0, \quad k=1,\ldots,n,\quad i,j=1,\ldots,4.
\end{equation}
If $\mathcal{B}$ is a generalized Doran-Morgan basis, the solution $\mathrm{S}$ is unique up to scalar (see \cite{Chmiel1}). If $\mathcal{B}$ is as in section \ref{sec:mongr}, solving the system (\ref{eq:system}) shows the solution $\mathrm{S}$ is also unique up to scalar. The bilinear form $\omega(x,y):=x^T\cdot\mathrm{S}\cdot y$ is a $G$-invariant symplectic form on $\mathbb{Q}^4\simeq\mathcal{S}_\mathbb{Q}$.

Let
$$
\Omega:=
\begin{pmatrix}
0&\operatorname{Id}_2\\
-\operatorname{Id}_2&0
\end{pmatrix}
$$
be the matrix of standard symplectic form invariant under $\mathrm{Sp}(4,\mathbb{Z})$ and let $T=\left(t_{i,j}\right)_{i,j=1,\ldots,4}$ be a generic invertible matrix. Treating coefficients of $T$ as indeterminates we solve the system of non-linear equations
\begin{equation}\label{eq:form}
T^T\cdot \mathrm{S}\cdot T=\Omega
\end{equation}
Let $\mathrm{T}\in\mathrm{GL}(4,\mathbb{Q})$ be a solution and put $\mathcal{B}_s:=\mathrm{T}\cdot\mathcal{B}$. Then $Mon^{\mathcal{B}_s}(\mathcal{P})\subset\mathrm{Sp}(4,\mathbb{Q})$.

A solution $\mathrm{T}$ realizes the monodromy group as a subgroup of $\mathrm{Sp}(4,\mathbb{Q})$ but not necessarily $\mathrm{Sp}(4,\mathbb{Z})$. The symplectic rational subspace $\mathcal{S}_\mathbb{Q}$ is uniquely determined by the conifold period, monodromy-invariance and irreducibility of the monodromy action. To ignore the issues of proving that our symplectic basis spans $\mathcal{S}_\mathbb{Z}$ we will use the following lemma:

\begin{lemma}
Let $G\subset\mathrm{Sp}(4,\mathbb{Z})$ be a subgroup which is Zariski-dense, resp. of finite index. If for some $M\in\mathrm{GL}(4,\mathbb{Q})$ we have $MGM^{-1}\subset\mathrm{Sp}(4,\mathbb{Z})$, then $MGM^{-1}$ is Zariski-dense, resp. of finite index, in $\mathrm{Sp}(4,\mathbb{Z})$.
\end{lemma}

\begin{proof}
First assume that $G$ is Zariski-dense. Since $MGM^{-1}$ is Zariski-dense in $\mathrm{Sp}(4,\mathbb{Z})$ if and only if $MGM^{-1}$ is Zariski-dense in $\mathrm{Sp}(4,\mathbb{Q})$, we will prove the latter. By assumption $G$ is Zariski-dense in $\mathrm{Sp}(4,\mathbb{Z})$, hence the standard symplectic form $\Omega$ is the unique (up to scaling) $G$-invariant skew-symmetric form. Thus $M\Omega M^T$ is the unique (up to scaling) $MGM^{-1}$-invariant skew-symmetric form. But $MGM^{-1}$ is a subgroup of $\mathrm{Sp}(4,\mathbb{Z})$ and so preserves $\Omega$, which implies that $M\Omega M^T=d\cdot\Omega$ for some $d\in\mathbb{Q}^*$. Thus $M$ is in the normalizer of $\mathrm{Sp}(4,\mathbb{Q})$. Since $G$ is Zariski-dense in $\mathrm{Sp}(4,\mathbb{Q})$, $MGM^{-1}$ is Zariski-dense in $M\mathrm{Sp}(4,\mathbb{Q})M^{-1}=\mathrm{Sp}(4,\mathbb{Q})$. The case when $G$ is of finite index follows from Theorem 6 in \cite{Borel}.
\end{proof}

The lemma shows that Zariski-density, resp. arithmecity, of the monodromy group $Mon(\mathcal{P})$ is independent of the choice of basis of $\mathcal{S}_\mathbb{Q}$, i.e. independent of a particular integral symplectic realization. Furthermore, to prove one of these properties for a subgroup $G\subset \mathrm{Sp}(4,\mathbb{Z})$, it is enough to prove it for any subgroup $H\subset G$.

Now we list certain subgroups $H:=H(\mathcal{P})$ of the monodromy groups $Mon(\mathcal{P})$ of double octic orphan operators. They are generated by powers of local monodromies. The list also contains transition matrices $T:=T^\mathcal{B}_{\mathcal{B}_s}$ between the basis $\mathcal{B}$ as in Section \ref{sec:mongr} and a symplectic basis $\mathcal{B}_s$ such that $H^{\mathcal{B}_s}\subset\mathrm{Sp}(4,\mathbb{Z})$.

\textbf{Operator 33}

$H:=\langle M_{1}^{2},M_{0}^{2^3},M_{2}\rangle$;\quad
$
T:=\begin{pmatrix}
	-4&0&192&-16\\
	-2&0&0&0\\
	1&8&-16&0\\
	0&0&8&0\\
\end{pmatrix}
$

\textbf{Operator 35}

$H:=\langle M_{0}^{2},M_{1},M_{-1}\rangle$;\quad
$
T:=\begin{pmatrix}
4&0&20&-4\\
-4&0&12&0\\
-1&-4&-1&0\\
0&0&-1&0\\
\end{pmatrix}
$

\pagebreak

\textbf{Operator 97}

$H:=\langle M_{0}^{2},M_{-2},M_{-1}\rangle$;\quad
$
T:=\begin{pmatrix}
	0&16&0&0\\
	8&0&0&4\\
	2&0&0&0\\
	0&0&1&0\\
\end{pmatrix}
$

\textbf{Operator 152}

$H:=\langle M_{0}^{2},M_{1}^{2},M_{-1}\rangle$;\quad
$
T:=\begin{pmatrix}
	0&32&0&0\\
	4&0&0&4\\
	1&0&0&0\\
	0&0&1&0\\
\end{pmatrix}
$

\textbf{Operator 153}

$H:=\langle M_{-1},M_{0},M_{-2}\rangle$;\quad
$
T:=\begin{pmatrix}
	0&96&-96&0\\
	40&0&-96&24\\
	8&0&0&0\\
	0&0&3&0\\
\end{pmatrix}
$

\textbf{Operator 243}

$H:=\langle M_{1}^{2^3},M_{2},M_{\frac{3}{2}}^{3},M_{\infty}^{2^5\cdot 3}\rangle$;\quad
$
T:=\begin{pmatrix}
-5120&-1920&-480&3200\\
0&-360&120&-160\\
128&48&0&-48\\
0&3&0&0\\
\end{pmatrix}
$

\textbf{Operator 248}

$H:=\langle M_{0},M_{-\frac{1}{2}},M_{-1}^{2^3},M_{-\frac{3}{2}}^{2},M_{-2}^{2^2}\rangle$;\quad
$
T:=\begin{pmatrix}
-8&0&256&-16\\
-4&0&0&0\\
1&8&-16&0\\
0&0&8&0\\
\end{pmatrix}
$

\textbf{Operator 250}

$H:=\langle M_{1},M_{\infty},M_{-1},M_{-2}\rangle$;\quad
$
T:=\begin{pmatrix}
32&-16&0&0\\
16&0&0&16\\
0&0&-2&0\\
-1&0&-1&0\\
\end{pmatrix}
$

\textbf{Operator 266}

$H:=\langle M_{-\frac{1}{4}},M_{-\frac{1}{2}}^{2^2\cdot3\cdot5^2},M_{-1},M_{0}^{2\cdot 3},M_{\frac{1}{2}}\rangle$;\quad
$
T:=\begin{pmatrix}
45&0&-24&15\\
0&0&-4&0\\
0&15&-3&0\\
15&0&0&0\\
\end{pmatrix}
$

\bigskip

To prove that the monodromy groups of double octic orphan operators are Zariski-dense in $\mathrm{Sp}(4,\mathbb{Z})$ we use the following criterion:

\begin{theorem}[Theorem 2.4, \cite{Rivin}]\label{th:criterium}
Let $G\subset\mathrm{Sp}(4,\mathbb{Z})$ be a subgroup. If there exists a prime $p\geq 5$ such that the reduction $\mod p$ map $\pi_p:G\rightarrow\mathrm{Sp}(4,\mathbb{Z}/p\mathbb{Z})$ is surjective, then $G$ is Zariski-dense in $\mathrm{Sp}(4,\mathbb{Z})$.
\end{theorem}

Let $\mathcal{P}$ be a double octic orphan operator and let $H^{\mathcal{B}_s}\subset\mathrm{Sp}(4,\mathbb{Z})$ be the symplectic realization of the subgroup $H\subset Mon(\mathcal{P})$ listed above. Given a subset $X\subset\mathrm{Sp}(4,\mathbb{Z}/p\mathbb{Z})$, it is a standard functionality in GAP to check whether the subgroup generated by $X$ is the entire symplectic group $\mathrm{Sp}(4,\mathbb{Z}/p\mathbb{Z})$. One verifies that the map $\pi_p:H^{\mathcal{B}_s}\rightarrow\mathrm{Sp}(4,\mathbb{Z}/p\mathbb{Z})$ is surjective for $p=7$. Using Theorem \ref{th:criterium} we obtain the following:

\begin{theorem}\label{th:dense}
The monodromy groups of double octic orphan operators are Zariski-dense in $\mathrm{Sp}(4,\mathbb{Z})$.
\end{theorem}

Theorem \ref{th:dense} shows that the question of arithmecity for double octic orphan operators is well-posed. We now prove that at least one of the associated monodromy groups is of finite index. Consider the double octic orphan operator $\textbf{250}$:

\bigskip

$\mathcal{P}_\textbf{250}$\ :=\ \(\Theta\, \left( \Theta-1 \right)  \left( \Theta -\frac12\right) ^{2}
+\frac18\,t\Theta\, \left( 44\,{\Theta}^{3}-96\,{\Theta}^{2}+65\,\Theta-12 \right) 
+{t}^{2} \left( \frac{19}2\,{\Theta}^{4}-23\,{\Theta}^{3}+{\frac {131}{8}}\,{\Theta}^{2}-{\frac {47}{8}}\,\Theta-\frac14 \right) \)

\ \ \ \ \ \ \ \ \ \ \(+{t}^{3} \left( \frac52\,{\Theta}^{4}-20\,{\Theta}^{3}-{\frac
	{23}{4}}\,\Theta-{\frac {17}{32}} \right)
-\frac1{32}\,{t}^{4} \left( 68\,{\Theta}^{2}+100\,\Theta+53 \right)  \left( 2\,\Theta+1 \right) ^{2}\)

\ \ \ \ \ \ \ \ \ \ \ \ \(-\frac14\,{t}^{5} \left( 8\,{\Theta}^{2}+14\,\Theta+9 \right)  \left( 2\,\Theta+1 \right) ^{2}
-\frac18\,{t}^{6} \left( 2\,\Theta+3 \right) ^{2} \left( 2\,\Theta+1 \right) ^{2}
\)

\bigskip

\noindent With respect to the symplectic basis $\mathcal{B}_s$ the monodromy group $M:=Mon^{\mathcal{B}_s}(\mathcal{P}_\textbf{250})$ is a subgroup of $\mathrm{Sp}(4,\mathbb{Z})$. It is Zariski-dense by Theorem \ref{th:dense}. It is generated by matrices of local monodromies around four finite singularities:

$$
\begin{pmatrix}
	1&0&0&0\\
	0&1&0&8\\
	0&0&1&0\\
	0&0&0&1\\
\end{pmatrix},\quad
\begin{pmatrix}
	-1&0&0&0\\
	0&1&0&4\\
	-2&0&-1&0\\
	0&-1&0&-3\\
\end{pmatrix},\quad
\begin{pmatrix}
	1&-2&-8&12\\
	-6&-5&-12&0\\
	2&3&9&-6\\
	1&2&6&-5\\
\end{pmatrix},\quad
\begin{pmatrix}
	5&4&4&8\\
	8&9&8&16\\
	-4&-4&-3&-8\\
	-4&-4&-4&-7\\
\end{pmatrix}.
$$

\begin{theorem}\label{th:arithmetic}
The monodromy group $Mon(\mathcal{P}_{\textbf{250}})$ is arithmetic and its index is $23592960=2^{19}\cdot3^2\cdot5$.
\end{theorem}

\begin{proof}
The symplectic group $\mathrm{Sp}(4,\mathbb{Z})$ is finitely presented with known presentation (see \cite{Behr}). Using a brute force algorithm we found words in the chosen generators of $\mathrm{Sp}(4,\mathbb{Z})$ which give the matrices above. We then used the Computer Algebra System system Magma \cite{Magma} to apply the Todd-Coxeter algorithm for the subgroup $M\subset\mathrm{Sp}(4,\mathbb{Z})$. The algorithm completed coset enumeration, proving arithmecity.
\end{proof}

For other double octic orphan operators this method of computing the index of their monodromy groups did not produce any results. Computations modulo powers of small primes suggest that both arithmetic and thin examples can be found among them.

\end{document}